\documentclass[a4paper,11pt]{amsart}

\usepackage[utf8]{inputenc}
\DeclareUnicodeCharacter{00A0}{ }
\usepackage{amssymb}
\usepackage{MnSymbol}
\usepackage{enumitem}

\usepackage{ae}
\usepackage[utf8]{inputenc}
\numberwithin{figure}{section}

\usepackage{mathrsfs,a4wide}
\usepackage{esint}
\usepackage{color}

\usepackage{xcolor}
\usepackage{comment}
\usepackage{tcolorbox}

\usepackage[colorlinks=true, linkcolor=red, urlcolor=red,citecolor=red]{hyperref}

\usepackage{color}

\usepackage{import}

\usepackage[leqno]{amsmath}

\numberwithin{figure}{section}

\newtheorem{theorem}{Theorem}[section]

\newtheorem{lemma}[theorem]{Lemma}
\newtheorem{proposition}[theorem]{Proposition}

\theoremstyle{definition}
\newtheorem{definition}[theorem]{Definition}

\newtheorem{remark}[theorem]{Remark}
\numberwithin{equation}{section}

\addtocontents{toc}{\setcounter{tocdepth}{1}}

\newcommand{\D}{\mathrm{D}}
\newcommand{\R}{\mathbb{R}}
\newcommand{\N}{\mathbb{N}}

\newcommand{\Ha}{\mathcal{H}}

\newcommand{\beq}{\begin{equation}}
\newcommand{\eeq}{\end{equation}}
\newcommand{\dist}{{\rm dist}}

\newcommand{\id}{\mathrm{id}}

\newcommand{\Div}{\operatorname{div}}
\newcommand{\pa}{\partial}

\newcommand{\medint}{-\kern -,375cm\int}
\newcommand{\medintinrigo}{-\kern -,315cm\int}

\renewcommand{\d}{\mathrm{d}}

\begin{document}

\title[Stationary sets of the mean curvature flow]{Stationary sets of the mean curvature flow with a forcing term}

\author{Vesa Julin}

\author{Joonas Niinikoski}

\keywords{}

\begin{abstract} 
We consider the flat flow solution to the mean curvature equation with  forcing in $\R^n$. 
Our main results states that tangential balls in $\R^n$ under a flat flow with a bounded forcing term will 
experience fattening, which  generalizes the result in \cite{FJM} from the planar case to higher dimensions. Then, as in the planar case, we  characterize stationary sets in $\R^n$ for
a constant forcing term as finite unions of equisize balls with mutually positive distance.
\end{abstract}

\maketitle



\section{Introduction}
In this article we consider the mean curvature flow (MCF) with a bounded forcing term for compact embedded hypersurfaces. By definition this is a family of embedded surfaces $(\Sigma_{t})_{t \in [0,\infty)}$ in $\R^n$, with initial set $\Sigma_0$, and which moves according to the law
\begin{equation} \label{flow}
V_t = -H_{\Sigma_t} + f(t) ,
\end{equation}
where $V_t$ is the normal velocity, $H_{\Sigma_t}$ the mean curvature and $f$ a bounded measurable function. It is well known that the flow may develop singularities for a smooth initial set  when $n \geq 3$ \cite{Gray1} and even in the plane  when $f \neq 0$ \cite{BP}. In order to define the flow over the singular times and in order to define it for rough initial sets, one may define a weak solution by using either the level set formulation \cite{CGG, ES}, the flat flow via the minimizing movement scheme \cite{ATW, LS} or  Brakke's varifold formulation \cite{Bra}. The main issue  is that there is no unique way to define the weak solution and the previous methods may give rise to a different solution. The level-set approach provides a unique \emph{function} which is a solution of the corresponding partial differential equation in the viscosity sense, but its level sets may have positive volume. We call this phenomenon  \emph{fattening}. De Giorgi's minimal and maximal barriers provide essentially the same solution as the level-set approach and in this context the fattening  means that the minimal and the maximal solution do not agree. The fattening may occur instantaneously if the initial set is not regular \cite{BP, ES} or after a finite time for regular initial sets \cite{BP}. In this work we consider the  flat flow of \eqref{flow}, which is a solution  obtained via the minimizing movement scheme as in \cite{ATW, LS}. The flat flow can be defined for rough embedded initial hypersurfaces which are boundaries of sets of finite perimeter. Therefore,  it is more natural in this context to define the flow for sets rather than surfaces.  If the initial set is smooth the flat flow agrees with the classical solution for a short time interval, but in case of fattening it is not clear if it is  uniquely defined. 

Here we study the fattening for the flat flow of  \eqref{flow} in the specific case when the initial set is a union of two  tangent balls. It is well known that in this case the level-set solution produces instantaneously fattening \cite{BP, GK}. We also mention the work  \cite{DLN} where the authors study the same setting but add randomness to the flow. For a general introduction to the topic we refer to \cite{B}.  In our main theorem we generalize the result in \cite{FJM} from the plane to $\R^n$ and prove that the flat flow instantaneously connects the two tangent balls with a thin neck which continues to grow at least for a short period of time.
\begin{theorem}
\label{thm1}
Let $E_0 \subset \R^n$, $n \geq 2$ , be a union of two tangential balls $B(x_1,r)$ and $B(x_2,r)$. Let $(E_t)_t$ be a flat flow with forcing term $f$, which is  bounded by $C_0 \in \R_+$, starting from $E_0$. There exist positive numbers $\delta$, $c_1$ and $c_2$ depending only on $n$, $r$ and $C_0$ such that
for every $t \in (0,\delta)$ the set $E_t$ contains   a dumbbell shaped simply connected set which again contains the balls $B(x_1,r - c_1t)$, $B(x_2, r-c_1t)$ and $B((x_1+x_2)/2,c_2t)$. 
\end{theorem}
We note that the above result immediately generalizes to the case when the two balls do not have the same radii. This follows from Theorem \ref{thm1} and a standard comparison argument (see Proposition \ref{comp}).

Theorem \ref{thm1} implies that a union of tangent balls cannot be a stationary set of the flow \eqref{flow}. Therefore,  we may use the characterization of critical points of the isoperimetric problem from \cite{DM} to characterize all  stationary points of the flow \eqref{flow}. 
\begin{theorem}
\label{thm2}
A bounded set of finite perimeter $E_0 \subset \R^n$, with $n\geq 2$, is a stationary set of the flow \eqref{flow} (see definition \ref{statdef})  with a positive constant forcing $\Lambda$ exactly when it is a finite union of balls of radius $r=(n-1)/\Lambda$ with mutually positive distance.
\end{theorem}

Let us finally mention a few words about the proof of Theorem \ref{thm1}. We begin the proof as in the planar case \cite{FJM} by showing that any discrete approximation of the flat flow creates at the first step a neck which connects the two balls. After this we need to show that this neck is growing until the time $\delta$.   In the planar case it is enough to  construct a single barrier set  to show that the neck is growing (see \cite[Proof of Theorem 1.1]{FJM}).  In the higher dimensional case we need to construct a family of comparison sets which, together with a delicate comparison argument, implies that the neck is growing.  The novelty of the proof  is the construction of this discrete barrier flow. Similar idea is used in \cite{GK} in the context of level set solutions. The main difference  is that in our case the flow is defined via time discretization.        



\section{Notation and preliminary results}

Let us introduce some basic concepts and notation. First, our standing assumption throughout the paper is that the dimension $n$ is at least two and for $x \in \R^n$ we use the decomposition $x=(x_1,x')$, where $x_1 \in \R$ and $x' \in \R^{n-1}$.
For a given set $E \subset \R^n$ the distance function $d_E : \R^n \rightarrow \R$ is given by 
$d_E(x)= \inf_{y \in E} |x-y|$ and further the signed distance function $\bar d_E : \R^n \rightarrow \R$ is defined as
\[
\bar d_E(x) = 
\begin{cases}
-d_E(x), &x\in E \\
d_E(x), &x \in \R^n\setminus E.
\end{cases}
\]
For the empty set we use the convention that its signed distance function is $\infty$ everywhere. If $E \subset \R^n$ is Lebesgue-measurable, we will denote its $n$-dimensional Lebesgue-measure by $|E|$.

For a set of finite perimeter $E \subset \R^n$ the notation $\pa^* E$ denotes its reduced boundary as usual. Recall, that 
then $\overline{\pa^*E}$ is the support of the corresponding Gauss-Green measure and the perimeter of
$E$ is given by $P(E) = \Ha^{n-1}(\pa^* E)$. If $E$ is $C^1$-regular we have $\pa^* E = \pa E$. Moreover, we may always assume $\pa E = \overline{\pa^*E}$. The measure theoretic outer unit normal is defined in $\pa^*E$ and we  denote it by
$\nu_E$. If $E$ is a $C^1$-set, then $\nu_E$ agrees with the classical outer unit normal of $E$.
Again, for every $C^1$-vector field $\Psi: \R^n \rightarrow \R^n$ the tangential differential at $x$ is defined as
\[
\D_\tau \Psi (x) = \D \Psi(x) (I - \nu_E(x)\otimes\nu_E(x))
\]
and the tangential divergence as $\Div_\tau \Psi = \text{Tr}(\D_\tau \Psi (x))$.

For an orientable $C^2$ -hypersurface $\Sigma \subset \R^n$, with orientation $\nu_\Sigma : \Sigma \rightarrow \pa B(0,1)$,
the corresponding mean curvature $H_\Sigma(x)$ at $x \in \Sigma$ is defined 
as the sum of the principal curvatures $k_1(x),\ldots,k_{n-1}(x)$. If $E \subset \R^n$ is a $C^2$ -set, then $H_E(x)$ for $x \in \pa E$ 
denotes $H_{\pa E} (x)$, with the orientation $\nu_E$, and we have the classical (surface) divergence theorem
\[
\int_{\pa E} \Div_\tau \Psi \ \d \Ha^{n-1} = \int_{\pa E} H_E\langle \Psi, \nu_E \rangle \ \d \Ha^{n-1}
\]
for every $\Psi \in C^1_0(\R^n;\R^n)$. In  general, we say that
a set of finite perimeter $E \subset \R^n$ has a \emph{distributional mean curvature} $H_E \in L^1(\pa^* E)$, if for every $\Psi \in C^1_0(\R^n;\R^n)$ it holds
\beq
\label{distrMC}
\int_{\pa^* E} \Div_\tau \Psi \ \d \Ha^{n-1} = \int_{\pa^* E} H_E\langle \Psi, \nu_E \rangle \ \d \Ha^{n-1}.
\eeq
Note that for $C^2$-regular sets the distributional mean curvature agrees with the classical mean curvature. Finally, we say that 
a set of finite perimeter $E \subset \R^n$ is \emph{critical}, if it has a constant distributional mean curvature. By  \cite[Theorem 1 ]{DM} we know that the critical sets are 
characterized as finite union of balls with equal radius and mutually disjoint interiors. As a consequence, we have the following convergence result, see  \cite[Corollary 2]{DM}.

\begin{theorem}
\label{convergence}
Let $(E_i)_{i=1}^\infty$ be a sequence of sets of finite perimeters  in $\R^n$ with distributional mean curvature $H_{E_i}$, $E \subset \R^n$  a set of finite perimeter with a positive volume and $\Lambda$ 
a positive constant such that 
$|E \Delta E_i| \rightarrow 0$, $P(E_i) \rightarrow P(E)$ and $H_{E_i} \rightarrow \Lambda$ in the distributional sense, i.e., for every $\Psi \in C^1_0(\R^n;\R^n)$ it holds
\[
\lim_{i\rightarrow \infty} \int_{\pa^* E_i}  \Div_\tau \Psi  -\Lambda \langle \Psi, \nu_{E_i} \rangle \ \d \Ha^{n-1} = 0.
\]
Then $E$ is a finite union of balls with the equal radius $r= (n-1)/\Lambda$ and the balls have mutually disjoint interiors.
\end{theorem}

We will use \emph{solid of revolutions} which are obtained  by rotating a non-negative function around the $x_1$-axis
in $\R^n$. If $g$ is a  non-negative function defined on an interval $[a,b]$, we will denote by $\mathbf C(g,[a,b])$ the solid of 
revolution 
\[
\mathbf C(g,[a,b]) : = \{x \in \R^n : x_1 \in [a,b], \ x' \in \bar B^{n-1}(0,g(x_1))\}.
\] 
Again, by the \emph{heads} of $\mathbf C(g,[a,b])$  we mean the vertical part of the boundary
\[
\{x \in \R^n : x_1 \in \{a,b\}, \ x' \in \bar B^{n-1}(0,g(x_1))\}.
\] 
In the special case of a cylinder, symmetric to the hyperplane $\{x_1=0\}$, i.e.,  $g \equiv R >0$  and $b=-a$, we simply denote 
$\mathbf C(R,a) = \mathbf C(R,[-a,a]) $.
In the case, where $g$ is continuous on $[a,b]$ and vanishes at the endpoints, we make the following technical observation.  
\begin{remark}
\label{radialsection}
Suppose that $g \in C([a,b])$ is non-negative with $g(a)=0=g(b)$ and denote 
$E=\mathbf C(g,[a,b])$. Then for every $x_1 \in \R$ the section $\bar d_E(x_1, \ \cdot \ ) : \R^{n-1} \rightarrow \R$ is 
radially symmetric function strictly increasing in radius.
\end{remark}
If $g \in C([a,b]) \cap C^2((a,b))$ and is strictly positive, then for the surface of revolution
\[
\Gamma = \{(x_1,x') \in \R^n : x_1 \in (a,b), x' \in \partial B^{n-1}(0,g(x_1))\}
\]
with the inside-out orientation of  $\mathbf C(g,[a,b])$ one computes
\beq
\label{revolutionMC}
H_\Gamma (x) = -\frac{g''(x_1)}{(1+g'(x_1)^2)^\frac32} + \frac{1}{(1+g'(x_1)^2)^\frac12}\frac{n-2}{g(x_1)}
\eeq
for every $x \in \Gamma$. 

 A solid of revolution $\mathbf C(g,[a,b])$ is an example of a \emph{Schwarz symmetric} set. Recall, that for every measurable  set $E \subset \R^n$  its Schwarz symmetrization, or $(n-1)$-dimensional Steiner symmetrization,   with respect  to a direction $e \in \partial B(0,1)$ is a measurable set  $E^*_e $ such that for every $t \in \R$ the section $\{ z \in \langle e \rangle^\perp : te + z \in E^*_e \}$ is an open $(n-1)$-dimensional ball centered at the origin and  it holds
\[
\Ha^{n-1}(\{ z \in \langle e \rangle^\perp : te + z \in E \}) = \Ha^{n-1}(\{ z \in \langle e \rangle^\perp : te + z \in E^*_e \}) .
\]
Note that $|E^*_e| = |E|$ and if $E$ is a set of finite perimeter, then $E^*_e$ is also a set of finite perimeter and $P(E^*_e) \leq P(E)$ \cite{BCF}. A set $E$  is Schwarz symmetric with respect to $e$ if it holds $E^*_e = E$, up to a set of measure zero.

\section{Flat flows with forcing and stationary sets}

Let us first heuristically explain how  a flat flow with a forcing term is obtained via the minimizing movement scheme. Let 
$C_0 \in \R_+$ be a fixed constant and let $f: [0,\infty) \rightarrow \R$ be a measurable function satisfying the condition
\beq 
\label{condition}
\sup_{t\geq 0} |f(t)| \leq C_0.
\eeq
 The function $f$ will act as a time dependent forcing term in the dynamics. Now if $E_0$ is a bounded set of finite perimeter, we define for every $0<h\leq1$ a sequence of a bounded sets of finite perimeter $(E^{h,k})_{k=0}^\infty$, so called \emph{approximative sequence}, inductively by setting
first $E^{h,0} = E_0$ and for $k=0,1,2,\ldots$ we set $E^{h,k+1}$ to be a minimizer of the functional
\beq
\label{functional1}
F \mapsto P(F) + \frac1h \int_{F} \bar d_{E^{h,k}} \ \d x - \bar f(h,k) |F|,
\eeq
where $\bar f(h,k) = \fint_{kh}^{(k+1)h} f (t)\ \d t$. Then we define an \emph{approximate flat flow}  $(E^h_t)_{t\geq 0}$
by setting
\beq
\label{approxflow}
E^h_t = E^{h,k} \ \ \text{for} \ \ kh \leq t < (k+1)h.
\eeq
If there is a subsequence $(h_k)_{k \in \N}$ with $h_k \rightarrow 0$ 
and a family of bounded sets of finite perimeter $(E_t)_{t\geq 0}$
such that $E^{h_k} \rightarrow E_t$ for every $t \geq 0$
in the $L^1$-sense, then  we call $(E_t)_{t\geq 0}$ a \emph{flat flow} with forcing $f$ starting from $E_0$.
An existence of such a cluster point is always guaranteed, see for instance  \cite[Proposition 2.3]{FJM}.

Let us next make the above argument more precise by using the results in \cite{FJM, MSS}. We note that in \cite{MSS} the authors consider flat flow for volume preserving mean curvature flow, but the arguments will remain valid in our setting. 
Our first observation is that the functional in  \eqref{functional1}
may change its values if we perturb the set $E^{h,k}$ by  a set of measure zero due to the distance function. 
In order to use the notion of distance function consistently we define the class
\[
X_n = \{E \subset \R^n: E \ \text{is a bounded set of finite perimeter with} \ \pa E = \overline{\pa^* E}\}.
\]
Recall, that  every (essentially) bounded set of finite perimeter has a $L^1$-equivalent set from  $X_n$.
For given $0<h \leq 1$ and $\Lambda \in [-C_0,C_0]$ we define the functional 
\[
\mathcal F_{h,\Lambda} : X_n \times X_n \rightarrow \R \cup \{\infty\}
\]
by setting
\beq
\label{functional2}
\mathcal F_{h,\Lambda} (F,E) = P(F) + \frac1h \int_F \bar d_E \ \d x - \Lambda |F|.
\eeq
For every $E \in X_n$ the functional $\mathcal F_{h,\Lambda}( \ \cdot \ , E)$ admits a minimizer $E_{\min} \in X_n$,
see \cite[proof of Lemma 3.1]{MSS}.
If $E$ is empty, then $\bar d_E = \infty$ and hence necessarily $E_{\min}$ must be empty too. 
Minimizers have the following distance property, see the proof of \cite[Lemma 2.1]{LS} (or \cite[Proposition 3.2]{MSS}).
There is a positive constant  $\gamma = \gamma(n,C_0)$ such that for every $E \in X_n$ and every minimizer
$E_{\min} \in X_n$ of $ \mathcal F_{h,\Lambda}(\ \cdot \ , E)$ it holds
\beq
\label{distance}
|\bar d_E| \leq \gamma h^\frac12 \ \ \text{in} \ \ E \Delta E_{\min}.
\eeq
Now,  
\eqref{distance} has the following consequence. 
\begin{remark}
\label{union}
Suppose that $E_1,E_2, \ldots,E_k \in X_n$ have a mutually positive distance of at least of $d$. There is a positive $h_d = h_d(n,C_0,d) \leq 1$ such that for any $h \leq h_d$  it holds that any minimizer of
$ \mathcal F_{h,\Lambda}(\ \cdot \ , \bigcup_i E_i)$ must be a union of minimizers of  $\mathcal F_{h,\Lambda}(\ \cdot \ , E_i)$.
\end{remark}
In general, uniqueness of a minimizer of $\mathcal F_{h,\Lambda} (\cdot ,E)$ is not known. However,
the following weak comparison principle holds, see \cite[proof of Lemma 7.2]{CMP}.
\begin{proposition}
\label{comp}
Let $E,E' \in X_n$ and $\Lambda,\Lambda' \in [-C_0,C_0]$, with $\Lambda > \Lambda'$.
\begin{itemize}
\item[(i)] If $E' \subset\subset E$ and $E_{\min},E'_{\min} \in X_n$ are minimizers of 
$\mathcal F_{h,\Lambda}(\ \cdot \ , E)$ and $\mathcal F_{h,\Lambda}(\ \cdot \ , E')$ respectively, then
$|E'_{\min}\setminus E_{\min}| = 0$. 
\item[(ii)] If $E' \subset E$ and $E_{\min},E'_{\min} \in X_n$ are minimizers of 
$\mathcal F_{h,\Lambda}(\ \cdot \ , E)$ and $\mathcal F_{h,\Lambda'}(\ \cdot \ , E')$ respectively, then
$|E'_{\min}\setminus E_{\min}| = 0$. 
\end{itemize}
\end{proposition}
Concerning the regularity of a minimizer $E_{\min}$ of \eqref{functional2}, it is not difficult to see that it is a $(\Lambda_0,r_0)$-perimeter minimizer (using the notation from \cite{Ma}) with suitable $\Lambda_0,r_0 \in \R_+$ satisfying $\Lambda_0r_0 \leq 1$. Then it follows \cite[Theorem 26.5 and Theorem 28.1]{Ma} that $\partial^*E_{\min}$ is relatively open in $\pa E_{\min}$,  $C^{1,\alpha}$-regular hypersurface for every $0<\alpha<1/2$ and the (closed) singular part $\pa E \setminus \pa^* E$ has Hausdorff-dimension at most $n-8$. In particular, from now on \emph{we will use 
the convention that the minimizers are always open sets.}

Moreover, by considering local variations $(\Phi_t)_t$ of the form $\Phi_t = \id +t\Psi$, with $\Psi \in C^1_0(\R^n,\R^n)$,
and differentiating $t \mapsto \mathcal F_{h,\Lambda}(\Phi_t(E_{\min}), E)$ at zero we see that $E_{\min}$ has distributional mean curvature $H_{E_{\min}}$ which satisfies the 
Euler-Lagrange equation in the distributional sense
\beq
\label{EulerLagrange}
\frac{\bar d_E}{h} = - H_{E_{\min}}+ \Lambda \ \ \text{on} \ \ \pa^*E_{\min}.  
\eeq
Since $H_{E_{\min}}$ is  Lipschitz continuous on $\pa^*E$, then  by standard elliptic estimates $\pa^*E_{\min}$ is $C^{2,\alpha}$-regular
and \eqref{EulerLagrange} holds in the classical sense on the reduced boundary. In particular, $E_{\min}$ is a $C^{2,\alpha}$-set when $n \leq 7$. Finally, we note that if $\pa E$ satisfies an exterior or interior ball condition at $x$, then $x$ must belong to the regular  part $\pa^*E_{\min}$. This follows essentially from  \cite[Lemma 3]{DM}.


The next proposition states the somewhat obvious fact  that for a ball $E= B(x,r)$ any non-empty minimizer of $\mathcal F_{h,\Lambda} (\cdot ,E)$ must be a concentric ball. 
\begin{proposition}
\label{ballminim}
For a ball $E= B(x,r)$ every minimizer \eqref{functional2} must be an open concentric ball or the empty set. 
There is a positive constant $h_0= h_0(n,C_0)\leq 1$ such that if 
$h \leq h_0$, then every ball $B(x,r)$, with $r \geq (n-1)/C_0$, has a concentric ball $B(x,r_{\min})$ as a unique
minimizer of $\mathcal F_{h,\Lambda}( \ \cdot \ , E)$ and it holds 
\beq
\label{difference}
r_{\min}- r =\left[\Lambda - \frac{n-1}{r} + \mathcal O (h)\right]h.
\eeq
In the case $\Lambda = (n-1)/r$, the error term $\mathcal O (h)$ vanishes and hence $r_{\min} = r$.
\end{proposition}

\begin{proof}
The first claim is easy to see  by using the isoperimetric inequality
and the fact that for a given non-zero volume $V$ an open ball of the volume $V$, centered at $x$, is a unique minimizer 
of the energy $\int_F \bar d_{B(x,r)} \ \d y$ among the open sets $F$ of the volume $V$. Again, by using \eqref{distance} we see that if $h$ is sufficiently small compared to the radius, then every minimizer must be 
non-empty and hence a concentric ball. Thus, the uniqueness and \eqref{difference} follow from \eqref{distance} and the Euler-Lagrange equation \eqref{EulerLagrange}.
\end{proof}

Let us denote the Schwarz symmetrization of $E$ with respect to $x_1$-axis simply   by $E^*$. As we mentioned above, Schwarz symmetrization decreases the perimeter and preserves the volume. Moreover, for a smooth set in the case of equality $P(E^*) = P(E)$ it holds that  every vertical slice $E_{x_1} = \{ x' \in \R^{n-1} : (x_1,x') \in E \}$ is $(n-1)$-dimensional ball  \cite{BCF}. We also  notice  that if the  set $E$  is Schwarz  symmetric with respect to $x_1$-axis, then  Schwarz  symmetrization also decreases the dissipation term of $\mathcal{F}_{h,\Lambda}(\ \cdot \ , E)$ defined  in \eqref{functional2}. This follows rather directly from Fubini's theorem.  
For a suitable solid of revolution $E$ around the $x_1$-axis, there is invariance of minimizers under the symmetrization.
\begin{proposition}
\label{symmetry}
If $E= \mathbf C(g,[a,b])$, with a non-negative and continuous $g$ attaining the zero value at the endpoints, 
then every (open) minimizer $F$ of $\mathcal{F}_{h,\Lambda}(\ \cdot \ , E)$ defined  in \eqref{functional2} is Schwarz symmetric with respect to $x_1$-axis.
\end{proposition}

\begin{proof}
 Let $F$ be a such a minimizer. We may assume $F$ to be non-empty. 
Now $P(F^*) \leq P(F)$, $|F^*| = |F|$. By Remark \ref{radialsection} every section $\bar d_E (x_1 ,\ \cdot \ )$ is radially symmetric and strictly increasing in radius which implies via Fubini's theorem that the $(n-1)$-dimensional Lebesgue measure of the symmetric difference $|(F^*)_{x_1} \Delta F_{x_1}|_{n-1}$ of the vertical slices $(F^*)_{x_1}$ and $F_{x_1}$ is zero  for almost every $x_1$,
since otherwise it would hold
\[
\int_{F^*} \bar d_E \ \d x <   \int_F \bar d_E \ \d x  
\]
and hence $\mathcal{F}_{h,\Lambda}(F^* , E)<\mathcal{F}_{h,\Lambda}(F, E)$ contradicting the minimality of $F$. Since $F$ is open, then every vertical slice $F_{x_1} \subset \R^{n-1}$ is open too and then the previous observation quarantees that 
$(F^*)_{x_1} = F_{x_1}$ for almost every $x_1$. Thus, the openess of $F$ implies that the equality holds for every $x_1$.  

\end{proof}


After this discussion we are convinced that an approximative sequence $(E^{h,k})_{k=0}^\infty$, starting from $E_0$, where for every $k=1,2\ldots$, the set $E^{h,k+1}$ is defined as a minimimizer of the functional 
$\mathcal F_{h,\bar f(h,k)}( \ \cdot \ , E^{h,k})$ defined in \eqref{functional2}, is well-defined. Further, we may define the approximative flat flow $(E^h_t)_{t \geq 0}$ as in 
\eqref{approxflow}. We have for every $t \geq h$ that the set $E^h_t$ is open and $C^2$-regular up to a singular part 
$\pa E^h_t \setminus \pa^* E^h_t$ of Hausdorff-dimesion at most  $n-8$. Moreover, $E^h_t$, with $t \geq h$, has a 
distributional mean curvature $H_{E^t_h}$ which satisfies the Euler-Lagrange equation \eqref{EulerLagrange}, with $\Lambda = \bar f(h,\lfloor t/h \rfloor -1)$, in a weak sense and 
on $\partial^*E^h_t$ in the classical sense. For more properties of the approximative flat flows, when the forcing term satisfies \eqref{condition}, such as local Hölder continuity of $(t,s) \mapsto |E^h_t \Delta E^h_s|$ and perimeter control we refer to  \cite[Proposition 2.3]{FJM}.

Next, we define  \emph{stationary sets} of \eqref{flow} with constant forcing term by using  flat flows as in
\cite[Definition 3.1]{FJM}.
\begin{definition}
\label{statdef}
A non-empty set $E_0 \in X_n$ is a stationary set of \eqref{flow} for a constant  forcing term $f \equiv \Lambda >0$, if for any flat flow, starting from $E_0$ it holds
\[
\sup_{0\leq t \leq T} |E_t \Delta E_0| = 0 
\]
for every $T>0$.
\end{definition}

By using Remark \ref{union} and Proposition \ref{ballminim} one may conclude the obvious direction of Theorem \ref{thm2}, that is,
a finite union of equisize balls with a mutually positive distance is a stationary set for  the constant forcing term
$\Lambda = (n-1)/r$, where $r$ is the radius of the balls. In turn, the following lemma states that the converse is almost true, that is, a stationary set is also critical, i.e., a finite union of balls
with equal radius and mutually disjoint interiors.

\begin{lemma}
\label{statlemma}
Every stationary set $E_0 \subset \R^n$ for a positive constant forcing term $\Lambda$, is a finite union of balls of radius $r=(n-1)/\Lambda$ with mutually disjoint interiors. 
\end{lemma}
\begin{proof}
The lemma  is already established in the two-dimensional case in  \cite[Lemma 3.4]{FJM}.
Again, the proof of the general case is analogous to the proof of  \cite[Lemma 3.4]{FJM} 
with the only essential change is that we use  Theorem \ref{convergence} instead of  \cite[Lemma 3.2]{FJM}.
Therefore, we only sketch the proof. Besides Theorem \ref{convergence}, we also use some 
basic properties of approximate flat flows proven in  \cite{FJM}.

We begin by fixing  times $0<T_1<T_2$. Then by Definition \ref{statdef} and \cite[Proposition 2.3]{FJM}  we have a decreasing sequence $(h_i)_{i=1}^\infty$, with $0<h_i <1$ and $h_i$ converging to zero, 
such that the approximate flat flows $(E^{h_i}_t)_{t\geq0}$, with constant forcing $f \equiv \Lambda$ and starting from $E_0$ satisfy
\beq
\label{statlemma1}
\lim_{i \rightarrow \infty} \sup_{t \in [T_1,T_2]} |E_0 \Delta E^{h_i}_t| = 0.
\eeq
Moreover, since the forcing term is constant, it follows from the argument in the proof of \cite[Proposition 2.4]{FJM} that there is $C \in \R_+$, independent of $h$, such that for every $t \in [T_1,T_2]$ with $t >h$ it holds 
\beq
\label{statlemma2}
\int_h^t \int_{\pa^* E^h_s} |H_{E^h_s} - \Lambda|^2 \ \d \Ha^{n-1} \d s \leq C\left[(P(E_0)-P(E^h_t)) + (|E^h_t|-|E_0|)\right].
\eeq
Now \eqref{statlemma1} and \eqref{statlemma2}  imply
$
\limsup_{i \rightarrow \infty} P(E_t^{h_i}) \leq P(E_0),
$
for every $t \in [T_1,T_2]$. On the other hand, by \cite[Proposition 2.3 ]{FJM} there is a radius $R >0$, independent of $i$, such that $E^{h_i}_t \subset B(0,R)$ for every $t \in [0,T_2]$. Then by the lower semi-continuity of the perimeter and by the previous estimate we have  $P(E^{h_i}_t) \rightarrow P(E_0)$ for every $t \in [T_1,T_2]$. Thus,  \eqref{statlemma2} yields
\[
\lim_{i \rightarrow \infty} \int_{T_1}^{T_2} \int_{\pa^* E^{h_i}_t} |H_{E^{h_i}_t} - \Lambda|^2 \ \d \Ha^{n-1} \d t = 0
\]
and further, by the mean value theorem, we find times $t_i \in (T_1,T_2)$ such that
\[
\lim_{i \rightarrow \infty} \int_{\pa^* E^{h_i}_{t_i}} |H_{E^{h_i}_{t_i}} - \Lambda|^2 \ \d \Ha^{n-1}= 0.
\]
Since $P(E^{h_i}_{t_i})$ are uniformly bounded, we deduce by the previous estimate that   $H_{E^{h_i}_{t_i}} \rightarrow \Lambda$ in the distributional sense. We have also 
$|E_0 \Delta  E^{h_i}_{t_i}| \rightarrow 0 $ and $P(E^{h_i}_{t_i}) \rightarrow P(E_0)$ so the claim follows from  Theorem \ref{convergence}.
\end{proof}

Now, the non-trivial direction of Theorem \ref{thm2} is a rather straightforward consequence of  Lemma \ref{statlemma} and Theorem \ref{thm1}, since the latter guarantees that a critical set having two tangential balls
cannot be stationary. The reasoning is exactly same as in the planar case, but for the sake of completeness we sketch the argument here. 
To this aim, let $E_0$ be a finite union of balls with equal radius $r$ containing a union of two tangential balls, say 
\[
E'_0 =B(x_1,r) \cup B(x_2,r).
\]
Let $(E_t)_{t\geq 0}$ be any flat flow with a bounded forcing $f$ starting from $E_0$.
By applying the second claim of Proposition \ref{comp} we find a
a flat flow $(E'_t)_{t\geq 0}$ with the forcing $f-1$ starting from $E'_0$
such that $|E'_t \setminus E_t| = 0$ for every $t \geq 0$.
By Theorem \ref{thm1} we have
$|B((x_1+x_2)/2, ct) \setminus E'_t| = 0$, and thus  $|B((x_1+x_2)/2, ct) \setminus E_t| = 0$  for  some $c \in \R_+$ and  for all small $t>0$. On the other hand, it clearly holds
$|B((x_1+x_2)/2, ct) \setminus E_0| > 0$. Therefore, we deduce that  $|E_0 \Delta E_t|>0$ for all small $t>0$. Thus, $E_0$ cannot be stationary.


\section{Proof of Theorem \ref{thm1}}


\begin{proof}[Proof of Theorem \ref{thm1}]
Let $E_0$ be a union of two tangential  balls of radius $r$. We may assume that 
$E_0 = B(-re_1,r) \cup B(re_1,r)$. Recall that for $0<h \leq 1$ an approximative sequence 
$(E^{h,i})_{i=0}^\infty$ is defined recursively by setting first $E^{h,0} = E_0$ and for each $i=0,1,2,\ldots$
the set $E^{h,i+1}$ is chosen to be a minimizer of $\mathcal F_{h,\bar f(h,i)}( \ \cdot \ , E^{h,i})$ defined in \eqref{functional2}, where
$\bar f(h,i) = \fint_{ih}^{(i+1)h} f(t) \ \d t$. Recall also that  $f:[0,\infty) \rightarrow \R$ is measurable and satisfies  \eqref{condition}
(and hence $|\bar f(h,i)| \leq C_0$). Now each $E^{h,i}$, with $i \geq 1$, satisfies the Euler-Lagrange equation \eqref{EulerLagrange}
with the constant $\Lambda = \bar f (h,i-1)$. Again, the corresponding approximative flat flow $(E^h)_{t\geq 0}$  is given by \eqref{approxflow}.

Our aim is to show that for a
time interval $(0,\delta]$, with $\delta$ small enough, we may construct barrier sets $G^{h,i} \subset E^{h,i}$ for $i=1,\dots, \lfloor \delta / h \rfloor + 1$  such that for every $h \leq t \leq \delta$ the barrier $G^{h,\lfloor t / h \rfloor}$ contains  a simply connected set $A_t$   defined as 
\beq
\label{goal}
A_t = \mathbf C(c_1t,r) \cup B(-re_1,r-c_2t) \cup B(re_1,r-c_2t),
\eeq
with some $c_1,c_2 \in \R_+$, depending only on $n$, $r$ and $C_0$, provided that $h$ is small enough.
Now, if $(E_t)_{t\geq 0}$
is any cluster flow, then $A_t \subset E_t^h$ implies $|A_t \setminus E_t| = 0$ for every $t \in (0,\delta)$ and further, 
since $E_t \in X_n$, this means $\mathrm{int}(A_t) \subset E_t$. The rest of the claim follows trivially from this. 


We first note that it is easy to see that the balls $B(\pm re_1,r-c_2t)$ are contained in $E_t$. 
Indeed, by possibly 
replacing $C_0$ with $\max\{C_0,4(n-1)/r\}$, we may assume that $r/4\geq (n-1)/C_0$.
Then by (i) of Proposition \ref{comp} and Proposition \ref{ballminim}
we find $\eta = \eta(n,r/2,C_0) \in \R_+$ and $0<h_0 = h_0(n,r/2,C_0) \leq 1$ such that for every $0<h \leq h_0$
the following implication holds
\beq
\label{step01}
B(x,\tilde r) \subset E^{h,i} \ \text{with} \ \tilde r \geq r/2 \implies \bar B(x,\tilde r - \eta h) \subset E^{h,i+1}.
\eeq

We split the proof into three steps.

\medskip

\textbf{Step 1:} We prove that there is a positive $\alpha = \alpha(n,r,C_0)$, such that 
the set $E^{h,1}$ contains the cylinder $\mathbf C(\alpha h^\frac14,\alpha h^\frac12)$ provided that $h$ is sufficiently small. 

To this aim, let $\tau >0$ be a small number, which we will fix later.
We use $C$ and $c$ for positive constants which may change from line to line but always depend only on $n$, $r$ and $C_0$. We also use a further shorthand notation $\mathbf C_{h,\tau}$ for the cylinder $\mathbf C(\tau h^\frac14,\tau h^\frac12)$

 By \eqref{step01} the balls
$\bar B(\pm re_1, r-\eta h)$ are contained in  $E^{h,1}$ provided that $h$ is small enough.
Again, assuming $\tau \leq r/2$ and $h$ to be sufficiently small we have 
\begin{align*}
(r-\eta h)^2 -(r-\tau h^\frac12)^2 
&= (2r-\eta h - \tau h^\frac12)(\tau-\eta h^\frac12)h^\frac12 \\
&> \frac{\tau r}{2} h^\frac12  \geq \tau^2 h^\frac12.
\end{align*}
Thus the heads of the cylinder $\mathbf C_{h,\tau}$, which are the vertical parts of the boundary, 
are contained in $B(\pm re_1, r-\eta h)$ and therefore, in turn, in the set $E^{h,1}$. Since $\partial E^{h,1}$ is $C^2$ (possibly
up to a closed singular part of Hausdorff-dimension at most $n-8$), then by a foliation and continuity argument we may assume that 
$\Ha^{n-1}(\partial \mathbf C_{h,\tau} \cap \partial E^{h,1})=0$. Otherwise, we would choose $\tau/2 \leq \tilde \tau < \tau$ such 
that the heads of the cylinder $\mathbf C_{h,\tilde \tau}$ are contained in $B(\pm re_1, r-\eta h)$ and $\Ha^{n-1}(\partial \mathbf C_{h,\tilde\tau} \cap \partial E^{h,1})=0$. This implies
\beq
\label{step11}
P (\mathbf C_{h,\tau} \cup E^{h,1}) = \Ha^{n-1}(\partial \mathbf C_{h,\tau} \setminus E^{h,1}) + 
\Ha^{n-1}(\partial E^{h,1} \setminus \mathbf C_{h,\tau}). 
\eeq
Again, we have the following estimates 
\begin{align}
\label{step12}
\Ha^{n-1}(\partial  \mathbf C_{h,\tau}  \setminus E^{h,1}) 
&\leq C \tau^{n-1} h^\frac n4, \\
\label{step13}
\bar d_{E_0} &\leq C \tau^2 h^\frac12 \ \ \text{in} \ \ \mathbf C_{h,\tau}  \ \ \text{and} \\
\label{step14}
|\mathbf C_{h,\tau}| & \leq C\tau^n h^\frac{n+1}{4}.
\end{align}
We show that 
$
\mathbf C(\tau h^\frac14/2,  \tau h^\frac12)\subset E^{h,1}
$
which implies the claim of Step 1 by choosing $\alpha=\tau/2$.  Suppose by contradiction that this does not hold.  We first notice that $E_0 = \mathbf C(g,[-2r,2r])$ with a continuous $g$ having the zero value at the endpoints and therefore by Proposition \ref{symmetry} it holds $E^{h,1} = (E^{h,1})^*$, i.e., $E^{h,1}$ Schwarz symmetric with respect to $x_1$-axis. Using this and the fact that the heads of $\mathbf C_{h,\tau}$ are in $E^{h,1}$, we conclude 
\[
\Ha^{n-1}(\partial E^{h,1} \cap\mathbf C_{h,\tau}) \geq c \tau^{n-1}h^\frac{n-1}{4}.
\]
We use the set $\mathbf C_{h,\tau} \cup E^{h,1}$ as a competitor in the energy $\mathcal F_{h,\bar f(h,0)} ( \ \cdot \ ,E_0)$.  By using the previous estimate as well as \eqref{step11}, \eqref{step12}, \eqref{step13} and \eqref{step14} and assuming 
$h$ to be small enough we estimate
\begin{align*}
\mathcal F_{h,\bar f(h,0)} (\mathbf C_{h,\tau} \cup E^{h,1},E_0) 
&= \mathcal F_{h,\bar f(h,0)} (E^{h,1},E_0)+\Ha^{n-1}(\partial \mathbf C_{h,\tau} \setminus E^{h,1}) 
- \Ha^{n-1}(\partial E^{h,1} \cap \mathbf C_{h,\tau}) \\
&+\frac1h\int_{\mathbf C_{h,\tau} \setminus E^{h,1}} \bar d_{E_0} \ \d x -  \bar f(h,0)|\mathbf C_{h,\tau} \setminus E^{h,1}|\\
&\leq \mathcal F_{h,\bar f(h,0)} (E^{h,1},E_0)+ C \tau^{n-1} h^\frac n4 
-  c \tau^{n-1}h^\frac{n-1}{4} \\
&+ C(\tau^2 h^{-\frac12} +1)|\mathbf C_{h,\tau}|\\
&\leq \mathcal F_{h,\bar f(h,0)} (E^{h,1},E_0)+ C \tau^{n-1} h^\frac n4 
-  c \tau^{n-1}h^\frac{n-1}{4} + C \tau^{n+2} h^\frac{n-1}{4}\\
&\leq \mathcal F_{h,\bar f(h,0)} (E^{h,1},E_0)+ \frac{c}{2}\tau^{n-1} h^\frac{n-1}4 
-  c \tau^{n-1}h^\frac{n-1}{4} + C \tau^{n+2} h^\frac{n-1}{4}\\
&= \mathcal F_{h,\bar f(h,0)} (E^{h,1},E_0)+
\tau^{n-1}h^\frac{n-1}{4} \left(C\tau^2-\frac{c}{2}\right).
\end{align*}
Thus, by choosing $\tau < \sqrt{c/(2C)}$ we have 
$\mathcal F_{h,\bar f(h,0)} (\mathbf C_{h,\tau} \cup E^{h,1},E_0) <\mathcal F_{h,\bar f(h,0)} ( E^{h,1},E_0)$ which  contradicts the minimality of the set $E^{h,1}$. Hence, we  have $\mathbf C(\alpha h^\frac14,\alpha h^\frac12) \subset E^{h,1}$ for $\alpha = \tau/2$.

\bigskip

\textbf{Step 2:} 
We proceed by  constructing candidate family  for the  barrier sets $G^{h,i}$, for every $i=1,\dots, \lfloor \delta / h \rfloor + 1$ and small $\delta$, which  satisfy for every $h \leq t \leq \delta$  the condition $A_t \subset G^{h,\lfloor t / h \rfloor}$, where $A_t $ is defined in \eqref{goal}. To be more precise, we will define positive numbers $d_{h,i}$, $l_{h,i}$ and $r_{h,i}$ (such that $l_{h,i}$ increases and $r_{h,i}$ decreases linearly in discrete  time  and $r_{h,1} \rightarrow r$
as $h \rightarrow 0$) and  suitable convex and positive functions $\varphi_{h,i}:[-d_{h,i},d_{h,i}] \rightarrow \R$ with $ l_{h,i}/2 \leq \varphi_{h,i} \leq l_{h,i}$.
Then we define  the barrier sets $G^{h,i}$, see Figure \ref{fig}, as the union 
\[
G^{h,i} = \mathbf C(\varphi_{h,i},[-d_{h,i},d_{h,i}]) \cup \bar B(-re_1, r_{h,i}) \cup \bar B(re_1, r_{h,i}).
\]
\begin{figure}[h!]
\label{fig}
  \centering
    \includegraphics[scale=3]{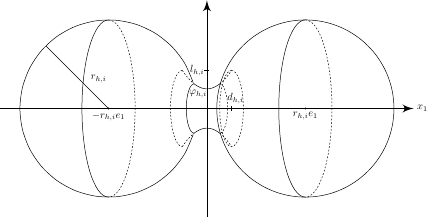}
\caption{A visualization of the barrier set $G^{h,i}$.}
\end{figure}
Here it follows from the selection of the parameters and the functions that the heads of the neck $\mathbf C(\varphi_{h,i},[-d_{h,i},d_{h,i}])$ are contained
in  the balls  $B(\pm re_1, r_{h,i})$ so $G^{h,i}$ will contain a simply connected set
\[
\mathbf C(l_{h,i}/2,r) \cup \bar B(-re_1, r_{h,i}) \cup \bar B(re_1, r_{h,i}).
\]
and hence the behavior of $l_{h,i}$ and $r_{h,i}$ yield the condition $A_t \subset G^{h,\lfloor t / h \rfloor}$.

To this end, let $0< \delta < 1$ be a sufficiently small number which will ultimately depend only on $n$, $r$ and $C_0$. 
Note that it holds   $hi \leq 2\delta$ for all $i=1,\dots, \lfloor \delta / h \rfloor + 1$ when $h$ is small. We begin by setting $r_{h,i} = r - \eta hi$.
Now $r_{h,i} \geq r - 2 \eta \delta$ so by assuming $\delta$ to be small enough we have $r_{h,i} \geq r/2$ and hence
thanks to \eqref{step01}
\beq
\label{step21}
\bar B(-re_1, r_{h,i}) \cup \bar B(re_1, r_{h,i}) \subset E^{h,i}.
\eeq
Again, set $\Lambda_0 = \max\{4\eta^2,2^9(n-2)^2,1\}$
and for each $i=1,\dots, \lfloor \delta / h \rfloor + 1$ define
\beq
\label{step21b}
 l_{h,i}= \Lambda_0h(i-1) + \alpha h^\frac14 \ \ \text{and} \ \ d_{h,i} = 2 \eta h(i-1) + \alpha h^\frac12. 
\eeq
It follows from the choice of $\Lambda_0$ that for $\delta$ small enough ($\delta \leq \Lambda_0^{-2}$) it holds
\beq
\label{step22}
\Lambda_0^\frac{1}{2}d_{h,i} \leq l_{h,i}.
\eeq
Moreover, $l_{h,i} \leq \Lambda_0 \delta + \alpha \delta^\frac14$ so by decreasing $\delta$ we may assume that  $d_{h,i}$ and $l_{h,i}$ are as small as we need.
Note that  by Step 1 we have 
\beq
\label{step23}
\mathbf C(l_{h,1}, d_{h,1}) \subset E^{h,1}.
\eeq
Further, by replacing $\alpha$ with $\min\{\alpha,r/4\}$, if necessary, we have
\begin{align*}
r_{h,i}^2 - (r-d_{h,i})^2 &= (r_{h,i} + r - d_{h,i})(r_{h,i} - r + d_{h,i}) \\
&= (2r - \eta h(3i-2) - \alpha h^\frac12)(\eta h (i-2) + \alpha h^\frac12 ) \\
&\geq r(\eta h (i-2) + \alpha h^\frac12 ) \\
&= r\eta h (i-1) + r(\alpha - \eta h^\frac12)  h^\frac12  \\
&\geq r\eta h (i-1) + \frac{\alpha r}{2}  h^\frac12  \\
&\geq 2\delta\Lambda_0^2 h (i-1) + 2\alpha^2  h^\frac12  \\
&\geq 2\Lambda_0^2 h^2 (i-1)^2 + 2\alpha^2  h^\frac12  \\
&\geq \left(\Lambda_0 h(i-1) + \alpha  h^\frac14\right)^2 = l_{h,i}^2,  \\
\end{align*}
when $\delta$ is small. Therefore, by Pythagorean theorem
\beq
\label{step24}
\{(\pm d_{h,i},x') \in \R^n : x' \in \bar B^{n-1}(0,l_{h,i})\} \subset \bar B(-re_1, r_{h,i}) \cup \bar B(re_1, r_{h,i}),
\eeq
i.e., the heads of the cylinder $\mathbf C(l_{h,i}, d_{h,i})$ are contained in the balls $\bar B(\pm re_1, r_{h,i})$.

We define for each $i=1,\dots, \lfloor \delta / h \rfloor + 1$   convex function
$
\varphi_{h,i} : [-d_{h,i+1},d_{h,i+1}] \rightarrow \R
$
by setting
\[
\varphi_{h,i}(t)= \frac{a_{h,i}}{2}(t^2-d_{h,i}^2) + l_{h,i},
\]
where $a_{h,i} = \Lambda_0^\frac12 / l_{h,i}$. Note that by \eqref{step22} we have $a_{h,i}d_{h,i} \leq 1$ and further $a_{h,i}d_{h,i}^2 \leq l_{h,i} $. Thus, $\varphi_{h,i}$ is $1$-Lipschitz and 
\beq
\label{step25}
\varphi_{h,i} \geq \varphi_{h,i}(0) = l_{h,i} - \frac{1}{2}a_{h,i}d_{h,i}^2 \geq \frac{l_{h,i}}{2}.
\eeq
Recall, that we set $G^{h,i}$ as the union
\[
G^{h,i} = \mathbf C(\varphi_{h,i},[-d_{h,i},d_{h,i}]) \cup \bar B(-re_1, r_{h,i}) \cup \bar B(re_1, r_{h,i}).
\]
By
 \eqref{step24} and \eqref{step25} the barrier $G^{h,i}$ contains the simply-connected set
\[
\mathbf C(l_{h,i}/2, r) \cup B(-re_1, r_{h,i}) \cup B(re_1, r_{h,i})
\]
and thus recalling $l_{h,i} = \Lambda_0 h(i-1) + \alpha h^\frac14$ and $r_{h,i} = r - \eta hi$ we find $c_1,c_2 \in \R_+$, depending only on $n$, $r$ and $C_0$, such 
that for every $h \leq t \leq \delta$ the barrier $G^{h,\lfloor t / h \rfloor}$ contains the set $A_t$ defined in \eqref{goal} with the
constants $c_1,c_2$.


\bigskip

\textbf{Step 3:}
We finish the proof by showing that  each barrier  $G^{h,i}$ constructed in the previous step is actually contained in $E^{h,i}$.
First, we conclude from \eqref{step21b}, \eqref{step22} and $\Lambda_0\geq 2\eta$ that when $\delta$ is small enough then for every $i=2,\dots,\lfloor \delta/h \rfloor+1$ it holds
\beq
\label{step26}
|\varphi_{h,i}-\varphi_{h,i-1}| \leq 2\Lambda_0 h \ \ \text{on} \ \ [-d_{h,i-1},d_{h,i-1}].
\eeq
Further, using  the  fact that $\varphi_{h,i}$ is $1$-Lipschitz  and $r_{h,i-1}  \geq r_{h,i}$ we obtain
\beq
\label{step26b}
G^{h,i}  \subset  \{ x \in \R^n : \bar d_{G^{h,i-1}}(x) \leq 4\Lambda_0 h \}. 
\eeq
By \eqref{step21} and \eqref{step23} we have
$
G^{h,1} \subset E^{h,1}.
$
Thus, we argue by induction. Assume that for $i=2,\dots,\lfloor \delta/h \rfloor+1$ it holds $G^{h,i-1} \subset E^{h,i-1}$. 
By  \eqref{step21b} and \eqref{step25} we have for small $\delta$
\[
\varphi_{h,i} - 2\Lambda_0 h \geq \frac{l_{h,i}}{2} - 2\Lambda_0 h \geq  \frac{\alpha}{2}h^\frac{1}{4} - 2\Lambda_0 h > 0
\]
and hence the set
$
\mathbf C(\varphi_{h,i} - 2\Lambda_0 h,[-d_{h,i},d_{h,i}])
$
is well-defined.
Again, by \eqref{step26} and $\varphi_{h,i} - 2\Lambda_0h \leq l_{h,i-1}$ it holds 
\beq
\label{step26c}
\mathbf C(\varphi_{h,i} - \Lambda_0 h,[-d_{h,i},d_{h,i}]) \subset G^{h,i-1}.
\eeq

Next, we define an auxiliary set  $\tilde G^{h,i} \subset  G^{h,i}$ as
\[
\tilde G^{h,i} = \mathbf C(\varphi_{h,i} - 2\Lambda_0 h,[-d_{h,i},d_{h,i}])  \cup \bar B(-re_1, r_{h,i}) \cup \bar B(re_1, r_{h,i}).
\]
Then by the induction assumption, \eqref{step26c} and  $r_{h,i-1}  \geq r_{h,i}$ we have $\tilde G^{h,i} \subset E^{h,i-1}$
and therefore by  \eqref{distance}
\beq
\label{step27}
\{x \in \tilde G^{h,i} : \dist(x,\partial \tilde G^{h,i}) > \gamma h^\frac12\} \subset E^{h,i}.
\eeq
Since the function $t \mapsto \max\{s: \{t\}\times \bar B^{n-1}(0,s) \subset \tilde G^{h,i}\}$ 
is increasing in $[0,r]$, decreasing in $[-r,0]$ and $ \tilde G^{h,i}$ is a solid of revolution, then for any 
\[
x \in  \mathbf C(\varphi_{h,i} - 2 \Lambda_0 h,[-d_{h,i},d_{h,i}]) \setminus (\bar B(-re_1, r_{h,i}) \cup \bar B(re_1, r_{h,i}))
\] 
the closest point $y \in \partial \tilde G^{h,i}$, i.e. $\dist(x,\partial \tilde G^{h,i}) = |x-y|$, must lie on 
\[
\partial \mathbf C(\varphi_{h,i} - 2\Lambda_0 h,[-d_{h,i},d_{h,i}])  \setminus \left(\bar B(-re_1, r_{h,i}) \cup \bar B(re_1, r_{h,i})\right). 
\]
Let us write $x = (x_1, x')$ and $y = (y_1,y')$. Since
 $\varphi_{h,i}$ is a $1$-Lipschitz function, then
\begin{align}
\notag
|\varphi_{h,i}(x_1)- |x'|| 
&\leq |\varphi_{h,i}(x_1)-\varphi_{h,i} (y_1)| + |\varphi_{h,i} (y_1)-|x'|| \\
\notag
&\leq |x_1-y_1| + ||y'|-|x'|| \\
\label{step27b}
&\leq 2|x-y| = 2\dist(x,\partial \tilde G^{h,i}).
\end{align}
We have $\gamma h^\frac12 > 2\Lambda_0 h$ and $\alpha h^\frac14 > 12 \gamma h^\frac12 $,
provided that $\delta$ is small, and thus by \eqref{step21b} and \eqref{step25}
\beq
\label{step28}
\frac{l_{h,i}}{4} \leq \varphi_{h,i} -3\gamma h^\frac12 < \varphi_{h,i} - 2\Lambda_0 h- 2\gamma h^\frac12  \ \ \text{on} \ \ [-d_{h,i},d_{h,i}].
\eeq
Therefore, it follows from \eqref{step21}, \eqref{step27}, \eqref{step27b} and \eqref{step28}  that
the set
$\mathbf C(\varphi_{h,i} - 3\gamma h^\frac12,[-d_{h,i},d_{h,i}])$
is well-defined and contained in $E^{h,i}$.

We argue by contradiction and assume that  $G^{h,i}$ is not 
contained in $E^{h,i}$. Since 
\[
\mathbf C(\varphi_{h,i} - 3\gamma h^\frac12,[-d_{h,i},d_{h,i}]) \subset E^{h,i},
\]
we may lift up the graph of $\varphi_{h,i} - 3\gamma h^\frac12$ until it touches the boundary $\partial E^{h,i}$. To be more precise,  by a continuity argument and \eqref{step21} there is $0<\tau < 3\gamma h^\frac12$ such that
\[
\mathbf C(\varphi_{h,i} - \tau,[-d_{h,i},d_{h,i}]) \subset E^{h,i} 
\]
and there is a point $z \in \Gamma \cap \partial E^{h,i}$, where
\[
\Gamma =  \{ (x_1,x') \in \R^n : x_1 \in (-d_{h,i},d_{h,i}) , x' \in \partial B^{n-1}(0,\varphi_{h,i}(x_1) -\tau)\}.
\]
In particular, the boundary $\partial E^{h,i}$ satisfies interior ball condition at $z$  and thus  $z$ belongs to the regular part of $\partial E^{h,i}$. Hence, by the comparison principle we have
$H_{ E^{h,i}}(z) \leq H_\Gamma(z)$, where $H_\Gamma$ is chosen to be compatible with the inside-out orientation of 
\[
\mathbf C(\varphi_{h,i} - \tau,[-d_{h,i},d_{h,i}]).
\]
Recalling \eqref{revolutionMC}, \eqref{step28}, $a_{h,i}d_{h,i} \leq 1$ and the choice of $\Lambda_0$ we estimate  
\begin{align*}
H_\Gamma(z)
&= - \frac{\varphi_{h,i}''(z_1)}{(1+(\varphi_{h,i}'(z_1))^2)^\frac32} + 
\frac{1}{(1+(\varphi_{h,i}'(z_1))^2)^\frac12} \frac{(n-2)}{\varphi_{h,i}(z_1)-\tau} \\
&= - \frac{a_{h,i}}{(1+(a_{h,i}z_1)^2)^\frac32} + \frac{1}{(1+(a_{h,i}z_1)^2)^\frac12}\frac{(n-2)}{\varphi_{h,i}(z_1)-\tau}  \\
&\leq - \frac{a_{h,i}}{2^\frac32} + \frac{4(n-2)}{l_{h,i}} \\
&= \frac{2^\frac72(n-2) - \Lambda_0^\frac12}{2^\frac32l_{h,i}} \leq -\frac{\Lambda_0^\frac12}{2^\frac52l_{h,i}}.
\end{align*}
Thus, by choosing $\delta$ to be small enough we have $H_\Gamma(z) \leq -(5\Lambda_0 + C_0)$. Then the Euler-Langrange equation
\eqref{EulerLagrange} for $E^{h,i}$ and $H_{ E^{h,i}}(z) \leq H_\Gamma(z)$ yield
\[
\bar d_{G^{h,i-1}}(z) = -H_{ E^{h,i}}(z)h + \bar f (h,i-1)h \geq 5\Lambda_0 h.
\]
However, by the construction we have $z \in G^{h,i}$ and thus the above  contradicts  \eqref{step26b}. Hence, we have $G^{h,i} \subset E^{h,i}$ for every $i=1,\dots, \lfloor \delta / h \rfloor + 1$. 
\end{proof}


\section*{Acknowledgments}
The research was supported by the Academy of Finland grant 314227. 

\end{document}